\documentclass[amsmath,secnumarabic,floatfix,amssymb,nofootinbib,nobibnotes,letterpaper,11pt,tightenlines]{revtex4}

\usepackage{times}

\usepackage{geometry}
\usepackage{amssymb}
\usepackage{latexsym, amsmath, amscd,amsthm}
\usepackage{graphicx}
\usepackage[percent]{overpic}
\usepackage{units}
\usepackage{hyperref}
\PassOptionsToPackage{caption=false}{subfig}
\usepackage[lofdepth]{subfig}
\usepackage{algorithm}
\usepackage{algorithmicx}
\usepackage{algpseudocode}

\newtheorem{theorem}{Theorem}

\newtheorem{proposition}[theorem]{Proposition}
\newtheorem{corollary}[theorem]{Corollary}

\theoremstyle{definition}
\newtheorem{definition}[theorem]{Definition}

\usepackage{units}
\newcommand{\R}{\mathbb{R}}
\newcommand{\I}{\mathbf{i}}
\newcommand{\Z}{\mathbb{Z}}

\newcommand{\dtheta}{\,\mathrm{d}\theta}
\newcommand{\dz}{\,\mathrm{d}z}
\newcommand{\dy}{\,\mathrm{d}y}
\newcommand{\sgn}{\operatorname{sgn}}

\newcommand{\dArea}{\,\mathrm{dArea}}
\newcommand{\Area}{\operatorname{Area}}
\newcommand{\RP}{\mathbb{RP}}

\begin{document}

\title[]{Random triangles and polygons in the plane}
\author{Jason Cantarella}
\altaffiliation{Department of Mathematics, University of Georgia, Athens GA}
\noaffiliation
\author{Tom Needham}
\altaffiliation{Department of Mathematics, The Ohio State University, Columbus OH}
\noaffiliation
\author{Clayton Shonkwiler}
\altaffiliation{Department of Mathematics, Colorado State University, Fort Collins CO}
\noaffiliation
\author{Gavin Stewart}
\altaffiliation{Courant Institute of Mathematical Sciences, New York University, New York NY}
\noaffiliation

\begin{abstract}
We consider the problem of finding the probability that a random triangle is obtuse, which was first raised by Lewis Caroll. Our investigation leads us to a natural correspondence between plane polygons and the Grassmann manifold of 2-planes in real $n$-space proposed by Allen Knutson and Jean-Claude Hausmann. This correspondence defines a natural probability measure on plane polygons. In these terms, we answer Caroll's question. We then explore the Grassmannian geometry of planar quadrilaterals, providing an answer to Sylvester's four-point problem, and describing explicitly the moduli space of unordered quadrilaterals. All of this provides a concrete introduction to a family of metrics used in shape classification and computer vision.
\end{abstract}

\maketitle

\noindent
\begin{flushright}
The issue of choosing a ``random triangle'' is indeed problematic. I believe the difficulty is explained in large measure by the fact that there seems to be no natural group of transitive transformations acting on the set of triangles.\\

\vspace{0.2in}

--Stephen Portnoy\\
\textit{A Lewis Carroll pillow problem: Probability of an obtuse triangle}\\
Statistical Science, 1994
\end{flushright}

In 1895, the mathematician Charles L.\ Dodgson, better known by his pseudonym Lewis Carroll, published a book of 72 mathematical puzzles called ``pillow problems'', which he claimed to have solved while lying in bed. The pillow problems mostly concern discrete probability, but there is a single problem in continuous probability in the collection: 
\begin{quotation}
\noindent Three points are taken at random on an infinite plane. Find the chance of their being the vertices of an obtuse-angled triangle. 
\end{quotation}
This is a very appealing problem and a number of authors have tackled it in the years since. After a moment's thought, it is clear that the main issue here is that the problem is ill-posed-- since there is no translation-invariant probability distribution on the infinite plane, the problem must really refer to a natural probability distribution on the space of triangles. But what probability distribution on triangle space is the right one?
Portnoy~\cite{Portnoy:1994vh} presented several different solutions to the problem involving distributions on triangle space invariant under various groups of transformations; Edelman and Strang~\cite{Edelman:2015td} connect the problem to random matrix theory and shape statistics; Guy~\cite{Guy:1993ep} got the answer 3/4 for a variety of measures, and the legendary statistician David Kendall got exact answers when the vertices of the triangle were chosen at random in a convex body~\cite{Kendall:1985ep}. Interestingly, Carroll himself gave a solution, but his method gives two different answers under the assumptions that side AB is the longest or second-longest side of the triangle! 

In fact, this problem has an even earlier history. In 1861, the actuary and editor of the \emph{Lady's and Gentleman's Diary} W. S. B. Woolhouse posed the same problem for triangles in space~\cite{Woolhouse:1861we}. Readers, including Stephen Watson~\cite{Watson:1862ur} came up with the answers rediscovered by Hall~\cite{Hall:1982id} for triangles whose vertices are uniformly chosen in a disk or ball. In 1865 Woolhouse posed a related problem~\cite{Woolhouse:1865vw}:
\begin{quotation}
	\noindent Three lines being drawn at random on a plane, determine the probability that they will form an acute triangle.
\end{quotation}

With its focus on the edges of the triangle rather than the vertices, this version of the problem is more closely related to the approach we present in this paper, which is based on a highly symmetric representation of triangle space as a Grassmann manifold~\cite{Cantarella:2014bl}. We will see that the Grassmannian picture of triangle space really does have a very natural group of transformations, that the pillow problem has a natural answer in our terms,\footnote{Our answer is different than the one Woolhouse arrived at! See~\cite{Woolhouse:1866uw}.} and that this entire story generalizes to the study of polygons with an arbitrary number of edges.

\section{Two paths to a construction}

We start by fixing notation. As is usual in triangle geometry; we let $A, B, C$ refer to the vertices of a triangle, $a$, $b$, and $c$ denote the lengths of the corresponding (opposite) sides, and use $\alpha$, $\beta$, $\gamma$ for the corresponding angles. 

We now construct a measure on triangle space. Since the geometry of a triangle is determined by $a,b,c$, it is immediately natural to want to assign a measure to the positive orthant of triples $a \geq 0, b \geq 0, c \geq 0$ in $\R^3$. But this space is not compact, so one is led to fix a scale for the triangle, by assuming\footnote{Why not perimeter $1$? The theory is the same either way, but if we make that choice, there will be many messy denominators to keep track of later on.} that the perimeter $a + b + c = 2$.  

Here we diverge from the beaten path. Since $a$, $b$, and $c$ must obey the triangle inequalities $a+b \geq c$, $b+c \geq a$, and $c+a \geq b$, the space of triangles is actually a triangular cone inside the positive orthant. This motivates us to write things in terms of the new variables
\begin{equation*}
s_a = \frac{-a + b + c}{2}, \quad s_b = \frac{a - b + c}{2}, \quad s_c = \frac{a + b - c}{2}.
\end{equation*}
These variables have a long history in triangle geometry. Perhaps most naturally if we construct three mutually tangent circles at the vertices of the triangle, their radii are $s_a$, $s_b$, and $s_c$. However, they recur in various other triangle formulae: 
in Heron's formula for the triangle area, or as trilinear coordinates for the Mittenpunkt or barycentric coordinates for the Nagel point of the triangle. They have a number of neat properties; for instance, the semiperimeter of the triangle is $s_a + s_b + s_c = \nicefrac{1}{2}(a + b + c)$, so we can again restrict to fixed perimeter by assuming that $(s_a,s_b,s_c)$ lies on the plane $x+y+z=1$. But now the space of all triangles is the entire orthant $s_a \geq 0, s_b \geq 0, s_c \geq 0$. 

We now introduce our last set of variables: we can parametrize fixed perimeter triangle space by the unit sphere $x^2 + y^2 + z^2 = 1$ if we let 
\begin{equation*}
x^2 = s_a, \quad y^2 = s_b, \quad z^2 = s_c.
\end{equation*}
This is actually an eightfold cover of triangle space, but that won't make any difference to our calculations in probability. We can now solve for $a$, $b$, and $c$ from the equations above.

\begin{definition}
\label{def:triangles}
The \emph{symmetric measure} $\mu$ on the space of perimeter 2 triangles is given by the pushforward of the uniform probability measure on the unit sphere under the map 
\begin{equation*}
a = 1 - x^2, \quad b = 1 - y^2, \quad c = 1-z^2.
\end{equation*}
\end{definition}

The variables $x$, $y$, and $z$ appear in various places in the theory of the triangle. We leave to the reader the (pleasant) proof of the following proposition:
\begin{proposition}
\label{prop:connections}
Various standard quantities in triangle geometry have natural expressions in terms of the coordinates $x$, $y$, and $z$. For triangles with unit semiperimeter, 
\begin{itemize}
\item The inradius $r$ and triangle area $A$ are both $|xyz|$.
\item The three exradii $r_1$, $r_2$, and $r_3$ are $\left|\frac{xy}{z}\right|$, $\left|\frac{yz}{x}\right|$ and $\left|\frac{xz}{y}\right|$.
\item The variables $|x|$, $|y|$ and $|z|$ are the (pairwise) geometric means of the exradii.
\end{itemize}
From this, it is easy to prove, for instance, the appealing triangle geometry theorem that $r r_1 r_2 r_3 = A^2$.
\end{proposition}

We have now defined a measure on triangle space, and it's clear from our construction that rotations of the sphere provide a beautiful, compact, transitive group of symmetries of triangles. This is already appealing, but one can immediately see that we have made various choices in the construction, and it is not clear how this construction would generalized to polygons with more edges. So now we start over and give another derivation of the same measure from a different point of view; this construction of polygon space is the one in our paper~\cite{Cantarella:2014bl} and is originally due to Knutson and Hausmann~\cite{Knutson:2_iyExxE}.

We will start by thinking of $\R^2$ as the complex plane. Since we are interested in polygons up to translation\footnote{We will deal with rotations shortly.}, we will represent a polygon by edges $e_1, \dots, e_n$, which are the complex numbers corresponding to the edge vectors. To fix perimeter, we want $|e_1| + \dots + |e_n| = 2$, but as before, we suspect that we will have more symmetries if we use the variables $z_1^2 = e_1, \dots, z_n^2 = e_n$ instead.

Now the polygon must close, so we are also imposing the condition $\sum e_i = 0$. If we write $z_i = u_i + v_i \I$, this condition becomes
\begin{equation}
0 = \sum z_i^2 = \sum (u_i^2 - v_i^2) + 2 u_i v_i \I.
\end{equation}
Rearranging, this is equivalent to $\sum u_i^2 = \sum v_i^2$ and $\sum u_i v_i = 0$. We have proved

\begin{proposition} 
Suppose $z_i = u_i + v_i \I$. The polygon with edges $z_1^2, \dots, z_n^2$ is closed and has perimeter 2 $\iff$ $\vec{u} = (u_1, \dots, u_n)$, $\vec{v} = (v_1, \dots, v_n)$ are orthonormal vectors in $\R^n$. 

Since squaring takes the $2^n$ points $(\pm z_1, \dots, \pm z_n)$ to the same edge set $(z_1^2, \dots, z_n^2)$, the Stiefel manifold of $V_2(\R^n)$ of orthonormal pairs of vectors in $\R^n$ is a $2^n$-fold cover of the space of polygons (up to translation) in the plane.\footnote{If one of the $z_i$ is zero, then $z_i=-z_i$ and the order of the cover is a lower power of 2, so this cover is actually branched over the points where some $z_i=0$, which correspond to the polygons with $i$th edge of length 0.}
\end{proposition}

An easy computation shows that rotating $(\vec{u},\vec{v})$ in the plane they span rotates the polygon (twice as fast) in the plane. This means that all pairs $(\vec{u},\vec{v})$ in the same plane give the same $n$-gon up to rotation, and so the space of $n$-gons up to translation and rotation is covered by the space of 2-planes in $\R^n$, which is the Grassmann manifold $G_2(\R^n)$.

\begin{definition}
\label{def:ngons}
The \emph{symmetric measure} $\mu$ on the space of $n$-gons with perimeter 2 is given by the pushforward of the uniform probability measure on the Grassmannian $G_2(\R^n)$ to polygon space under the map $P \mapsto e_i = (u_i + v_i \I)^2$, where $\vec{u}, \vec{v}$ are any orthonormal basis for the plane $P$.
\end{definition}

We have now given two definitions of the symmetric measure on triangle space, and must show they are the same. 

\begin{proposition}
The measures on triangle space of Definition~\ref{def:triangles} and Definition~\ref{def:ngons} are the same.
\end{proposition}

\begin{proof} We can identify $G_2(\R^3)$ with $G_1(\R^3)$ by taking normals to the planes; but $G_1(\R^3) = \R P^2$, which is double covered by $S^2$. The uniform measure on the Grassmannian is the pushforward of the uniform measure on this sphere. This means that both measures push forward from the standard measure on the sphere; we just need to check that a given point on the sphere maps to the same triangle under each construction.

If we take a point $\vec{p} = (x,y,z)$ on the sphere, the corresponding triangle is obtained using Definition~\ref{def:ngons} by completing $\vec{p}$ to a positive determinant orthonormal basis for $\R^3$ by adding the vectors $\vec{u}$ and $\vec{v}$. Any such choices of $\vec{u}$ and $\vec{v}$ will produce the same triangle shape since different choices will be related by a rotation of the triangle.

Given $\vec{u}$ and $\vec{v}$, the three vectors $\vec{p}$, $\vec{u}$, and $\vec{v}$ are the columns of an orthogonal matrix. The norm of each \emph{row} of the matrix is $1$, so for example $u_1^2 + v_1^2 = 1 - x^2$. But $u_1^2 + v_1^2 = \left|(u_1 + v_1 \I)^2 \right| = |e_1| = a$ is the length of the first side of the triangle, so we have shown that $a = 1 - x^2$. Similarly, $b = 1 - y^2$ and $c = 1 - z^2$, as they should be according to Definition~\ref{def:triangles}.
\end{proof}

\begin{definition}
When sidelength $c = 1 - z^2 \neq 0$, we can define a \emph{canonical triangle} associated to $\vec{p} = (x,y,z)$ by choosing our basis for the plane perpendicular to $\vec{p}$ to be
\begin{equation}
\left(
\begin{array}{c}
\vec{u} \\
\vec{v}
\end{array}
\right) =
\frac{1}{\sqrt{1 - z^2}}
\left(
\begin{array}{ccc}
 x z & y z & -x^2-y^2 \\
-y & x & 0 \\
\end{array}
\right)
\end{equation}
It is easy to check that $\vec{p}, \vec{u}, \vec{v}$ is a positively oriented orthonormal basis for $\R^3$. Continuing to unwind our definitions, if we place the center of edge $c$ at the origin, then edge $c = e_3$ points in the positive $x$-direction, and the vertices of the triangle are 
\begin{equation}
-\frac{1}{2}(1-z^2,0) \quad \frac{1}{2}(1-z^2,0) \quad \left(
\frac{-x^4+2 x^2 -2 y^2 + y^4}{2(1 - z^2)},-\frac{2xyz}{1 - z^2}
   \right).
   \label{eq:canonical triangle}
\end{equation}
\end{definition}
%Note: You can see the derivation of all this in rotation-to-triangle-family.nb, but I'm not adding it to the draft. 
This observation is helpful when performing computer experiments.

\section{A transitive group of isometries on polygon space}

We noted above that the Grassmannian $G_2(\R^n)$ has a uniform measure; this is simply the unique probability measure on $G_2(\R^n)$ which is invariant under the left action of $O(n)$ on $G_2(\R^n)$. But what does this action look like on triangle space? We can start by describing the action of $SO(3)$ (rotations) on the sphere of triangles in our coordinates above: rotating around a coordinate axis fixes a sidelength, and must therefore move the opposite vertex around an ellipse as the perimeter of the triangle is fixed:

\begin{proposition}
The circle $(\sqrt{1-z^2} \cos \theta,\sqrt{1-z^2} \sin \theta,z)$ formed by rotating a point on $S^2$ around the $z$-axis maps to a family of canonical triangles where vertices $A$ and $B$ are fixed and vertex $C$ follows the ellipse
\begin{equation}
C(\theta) = \left(\frac{1+z^2}{2} \cos 2\theta,-z \sin 2 \theta\right).
\end{equation} 
\end{proposition}

Note that the ellipse is parametrized \emph{clockwise} if $z > 0$ and counterclockwise only if $z < 0$, and that the map double-covers the ellipse. Moreover, this is the equal-area-in-equal-time parametrization of the ellipse, attesting to the naturality of this construction. The proof of the proposition is a pleasant exercise in plugging the parametrization of the circle into~\eqref{eq:canonical triangle} and simplifying.

There is a rather interesting open question here: what characterizes the family of triangles obtained by an arbitrary rotation of the sphere? Infinitesimal rotations of the sphere are linear combinations of coordinate axis rotations, so we know that the family is given by integrating infinitesimal linear combinations of the above elliptical vertex motions. The resulting pictures are certainly pretty (we show an example in Figure~\ref{fig:triangle rotation}), but we do not yet have a fully triangle-theoretic description of this family. 

\begin{figure}[htbp]
	\centering
		\includegraphics[scale=.4]{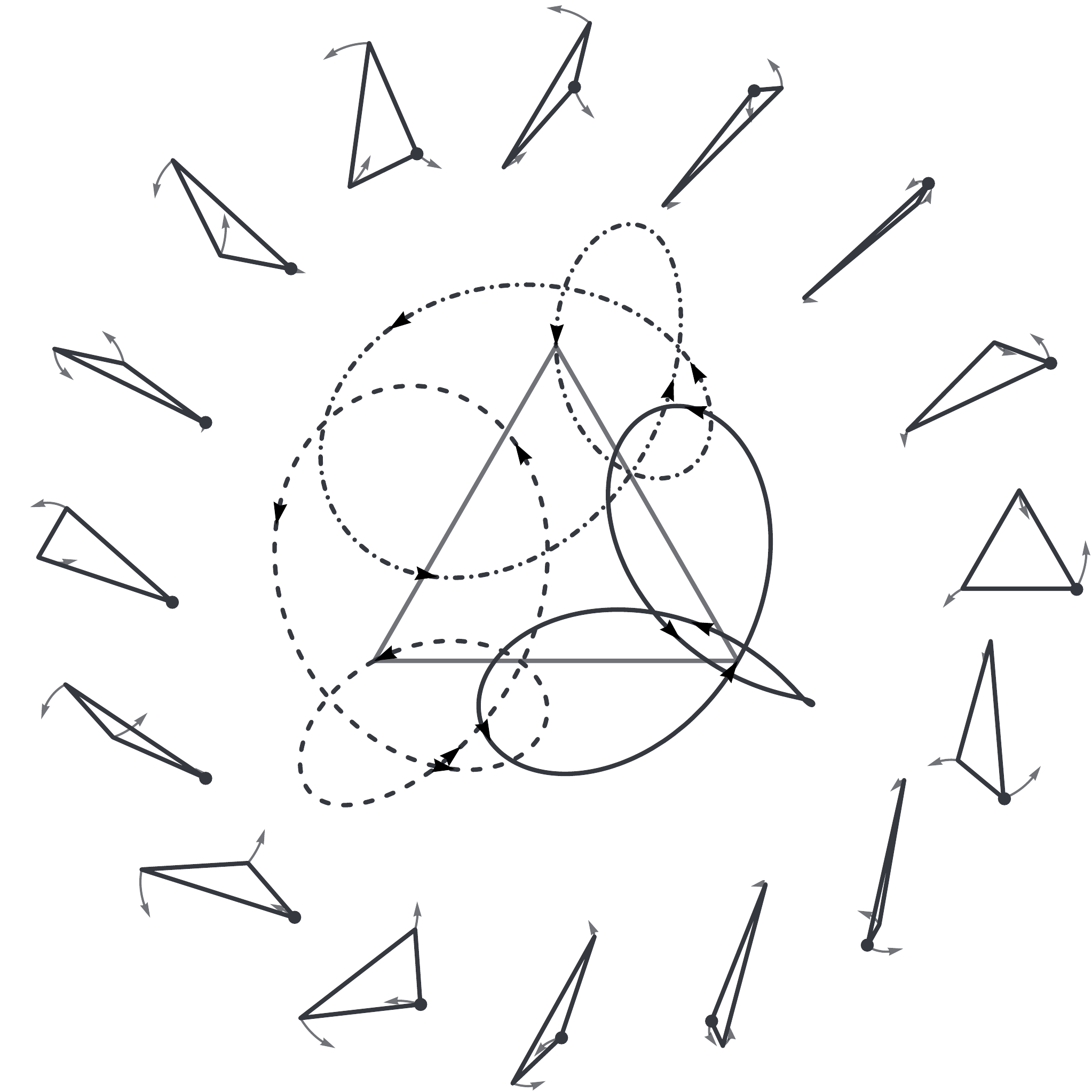}
	\caption{These are two different visualizations of the triangle motion induced by rotating the point $\frac{1}{\sqrt{3}}(1,1,1)$ corresponding to the equilateral triangle around the axis $\left(-1,1,-\sqrt{2}\right)$. The circle of triangles shows 16 equally-spaced points along the resulting great circle together with the path each vertex will traverse in the next time step. The figure in the middle shows the starting equilateral triangle along with the three curves traced out by the vertices. The solid curve is the path of the vertex marked with a dot in the outside triangles.}
	\label{fig:triangle rotation}
\end{figure}

\section{The Pillow Problem}

We are now in a position to answer Lewis Carroll's pillow problem, which boils down to identifying the obtuse triangles as a subset of the unit sphere and computing its area. Recall that a point $(x,y,z)$ on the unit sphere maps to the triangle with sidelengths 
\[
	a=1-x^2, \quad b=1-y^2, \quad c=1-z^2.
\]

The sphere can be split up into two (disconnected) regions: the acute triangles and the obtuse triangles; of course, the right triangles are the boundary between regions. But right triangles are easy to identify from the sidelengths: they are exactly the triangles such that $a^2+b^2=c^2$ or $b^2+c^2=a^2$ or $c^2+a^2=b^2$. Substituting in the above expressions for $a,b,c$ yields three quartics: 
\begin{equation}\label{eq:right triangle quartic}
	(1-x^2)^2+(1-y^2)^2=(1-z^2)^2
\end{equation}
and the two other cyclic permutations of the variables. The intersections of these quartics with the sphere give the collection of curves shown in Figure~\ref{fig:right triangle curves}.

\begin{figure}[htbp]
	\centering
		\includegraphics[height=2in]{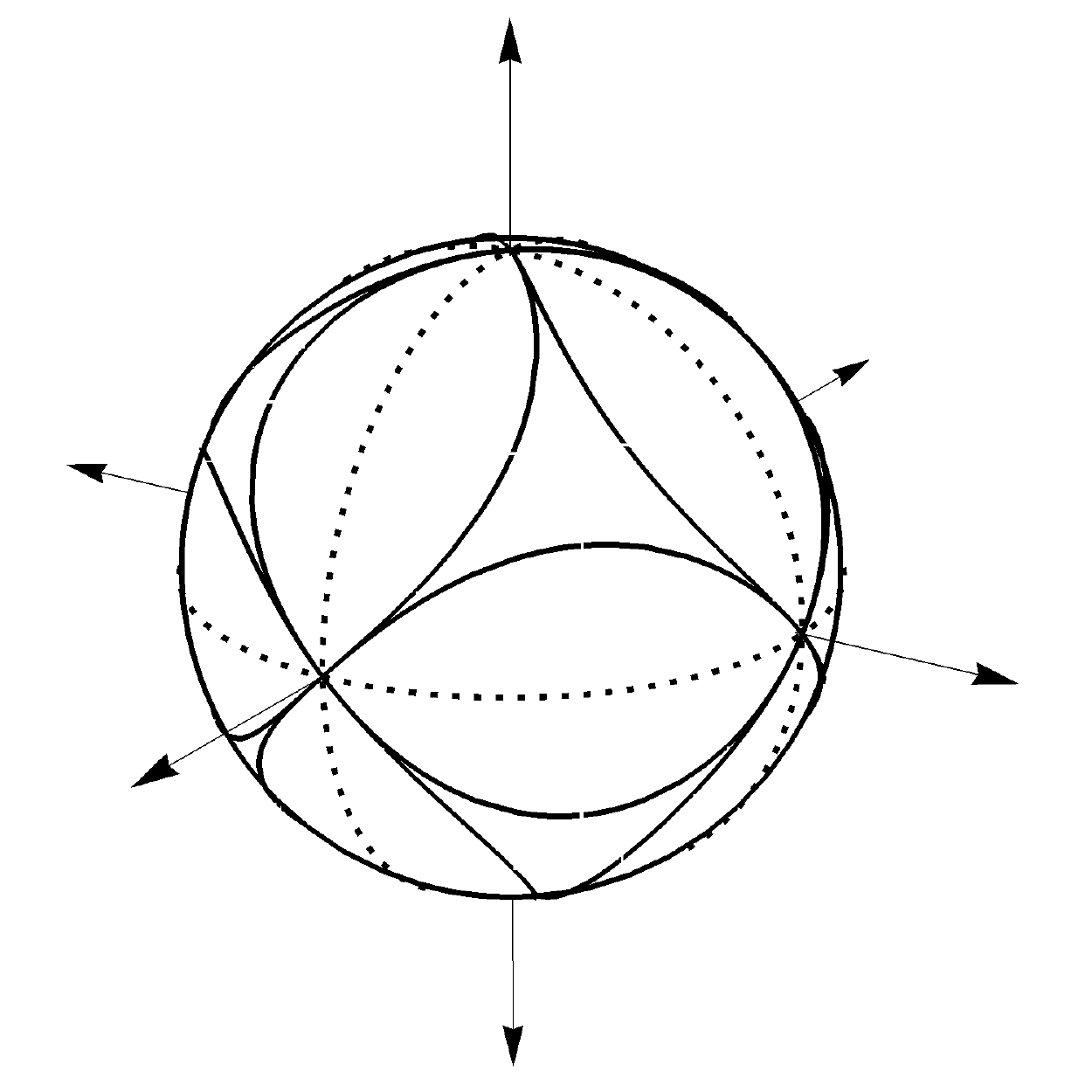}
		\includegraphics[height=2in]{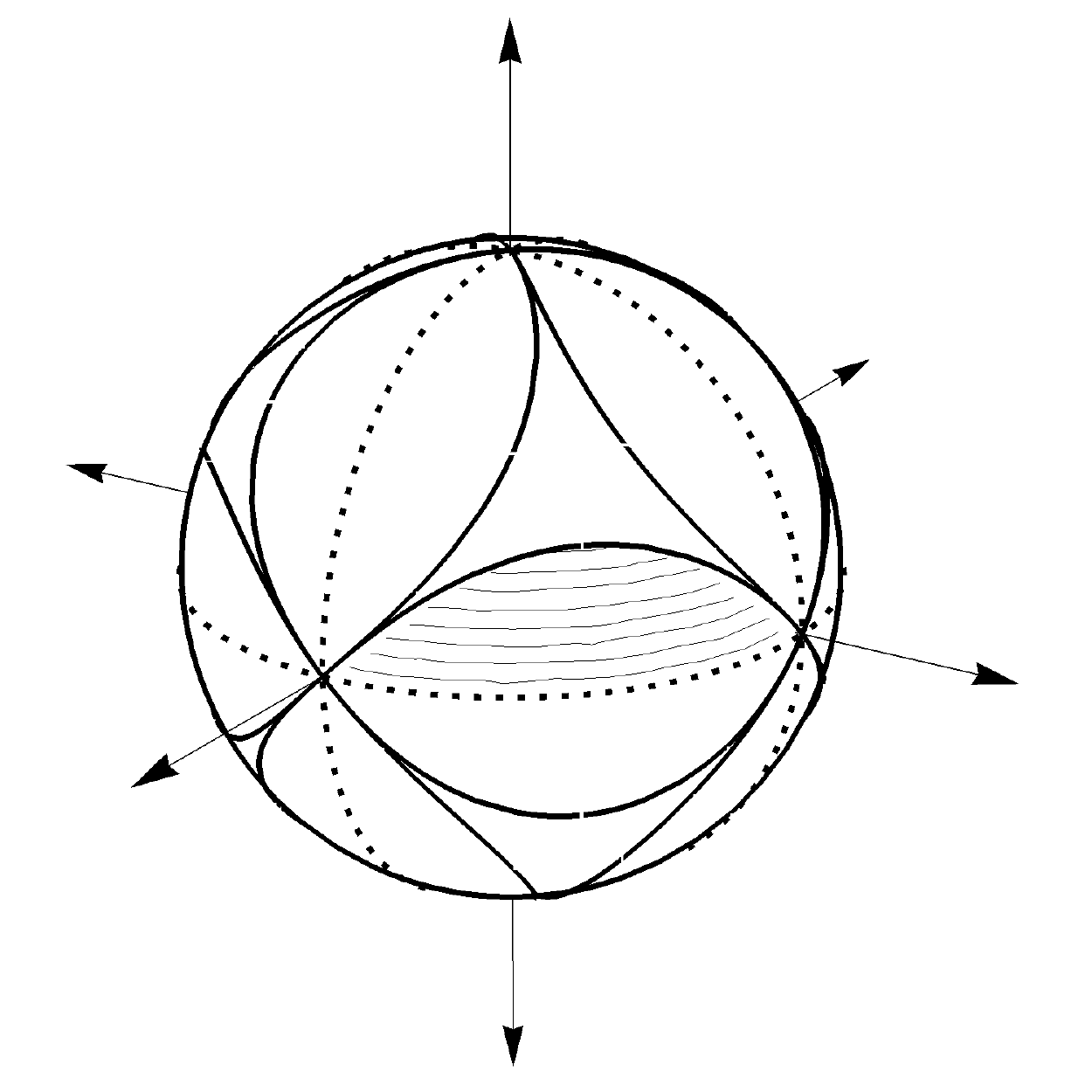}
	\caption{The right triangles are the heavy black curves on the sphere (the dotted lines indicate the intersections of the sphere with the coordinate planes). The hatched region in the right hand figure shows $\nicefrac{1}{24}$ of the region of obtuse triangles. We compute the area of this region below.}
	\label{fig:right triangle curves}
\end{figure}
%The figures here come from ellipse-fig-for-triangles-paper.nb

The sphere equation $x^2+y^2+z^2=1$ implies $z^2=1-x^2-y^2$, so \eqref{eq:right triangle quartic} can be re-written as
\[
	x^2+x^2y^2+y^2=1.
\]
Equivalently, $x^2=\frac{1-y^2}{1+y^2}$, which can be plugged into $z^2=1-x^2-y^2$ to get the following parametrization  for solutions of \eqref{eq:right triangle quartic}:
\begin{equation}\label{eq:right triangle curve parametrization}
	\left(\pm \sqrt{\frac{1-y^2}{1+y^2}},\pm y,\pm y \sqrt{\frac{1-y^2}{1+y^2}}\right).
\end{equation}

Computing in cylindrical coordinates, the area of the set of obtuse triangles is simply $24 \iint_R \dz \dtheta$, where $R$ is the hatched region shown at right in Figure~\ref{fig:right triangle curves}.

%\begin{figure}[htbp]
%	\centering
%		\includegraphics[height=2in]{obtuseregion.png}
%	\caption{$1/24$ of the obtuse triangles}
%	\label{fig:obtuse triangle region}
%\end{figure}

In turn, using Stokes' Theorem,
\[
	24 \iint_R \dz \dtheta = 24 \int_{\partial R} z \dtheta = 24\left( \int_{z=0} z \dtheta + \int_C z \dtheta \right),
\]
where $C$ is the upper boundary of the region parametrized by \eqref{eq:right triangle curve parametrization} with all signs positive. Of course, the first integral vanishes, so we are reduced to computing the second integral. Using \eqref{eq:right triangle curve parametrization} in conjunction with $\theta = \arctan(y/x)$ to simplify yields
\[
	24 \int_0^1 \left(\frac{2y}{1+y^4}-\frac{y}{1+y^2}\right)\dy.
\]
Both terms are easy to integrate using $u$-substitutions: the first is recognizably the derivative of $\arctan(y^2)$, while the second is the derivative of $-\frac{1}{2}\ln(1+y^2)$, so the area of the obtuse triangles is
\[
	24 \left[ \arctan(y^2) -\frac{1}{2}\ln(1+y^2)\right]_0^1 = 6\pi-12 \ln 2.
\]
Dividing by the area $4\pi$ of the sphere reveals the fraction of obtuse triangles to be exactly
\[
	\frac{3}{2}-\frac{3 \ln 2}{\pi} \approx 0.838093.
\]

%Triangle computations. Pushing measure forward to slack space, expected area. Solving Carroll's problemn explicitly. Mention connection to directional statistics on $S^2$ (Fisher/Kent Distribution)?

\section{Dirichlet distributions and expected areas of triangles}

Now that we have solved Carroll's problem, it's interesting to see what other expectations we can compute! Writing things in terms of the variables $s_a = \frac{-a+b+c}{2}$, $s_b$ and $s_c$, for instance, leads us to a really nice computation of the expectation of inradius (or area) and circumcurvature. Given a point $(x,y,z)$ on the unit sphere, the corresponding triangle has $s_a=x^2$, $s_b=y^2$, and $s_c=z^2$. Then $\varphi:(x,y,z)\mapsto (x^2,y^2,z^2)$ gives a map from the unit sphere $S^2$ to the simplex $\{s_a+s_b+s_c=1\}$.

\begin{proposition}\label{prop:pushforward dirichlet measure}
	The pushforward by $\varphi$ of the uniform measure on $S^2$ is the Dirichlet distribution on the simplex with concentration parameters $(\nicefrac{1}{2},\nicefrac{1}{2},\nicefrac{1}{2})$. This measure is $\frac{1}{\Area(\Delta ABC)} \,dx\,dy$.
\end{proposition}

\begin{proof}
	We will use $x$ and $y$ as coordinates on the simplex. In these coordinates, the simplex is parametrized by the triangle $x+y \leq 1$, $x\geq0$, $y\geq0$. The density of the uniform probability measure on $S^2$ (with respect to the standard area form), is just the constant function $\frac{1}{4\pi}$. But then, since $\varphi$ is an 8-to-1 map, the change-of-variables formula tells us that the density of the pushforward measure on the simplex $\Delta$ is 
	\begin{equation}\label{eq:pushforward density}
		8 \frac{1}{4\pi} \frac{1}{|J \varphi|},
	\end{equation}
	where $|J \varphi|$ is the Jacobian determinant of $\varphi$.
	
	Now, we can compute the Jacobian determinant by taking the square root of the determinant of the Gramian of the $3 \times 2$ matrix $\Phi = (\nabla \varphi_1 , \nabla \varphi_2)$, where $\nabla \varphi_i$ is the intrinsic gradient in $S^2$ of the coordinate function $\varphi_i$. Since
	\[
		\nabla \varphi_1 = \begin{pmatrix}
			2x \\ 0 \\ 0
		\end{pmatrix} - 2x \begin{pmatrix}
			x \\ y \\ z
		\end{pmatrix} \quad \text{ and } \quad \nabla\varphi_2 = \begin{pmatrix}
			0 \\ 2y \\ 0
		\end{pmatrix} - 2y \begin{pmatrix}
			x \\ y \\ z
		\end{pmatrix},
	\]
	it is straightforward to compute 
	\[
		|J \varphi|=\sqrt{\det \Phi^T \Phi} = 4|xyz|=4 \sqrt{s_a s_b s_c}.
	\]
	Combining this with \eqref{eq:pushforward density}, the density of the measure on $s_a + s_b + s_c = 1$ is
	\begin{equation}\label{eq:pushforward density 2}
		\psi(s_a,s_b)=\frac{1}{2\pi} s_a^{-\nicefrac{1}{2}} s_b^{-\nicefrac{1}{2}} s_c^{-\nicefrac{1}{2}},
	\end{equation}
	which is the density of the Dirichlet distribution, as claimed. Since we've fixed the semiperimeter $s=1$ for our triangles, we have $\sqrt{s_a s_b s_c} = \sqrt{s s_a s_b s_c}$, which Heron's formula says is the area of the triangle.
\end{proof}

\begin{corollary}\label{cor:expected area}
	The expected value of the area of a perimeter-2 triangle with respect to the symmetric measure is
	\[
		E(\text{Area}) = \frac{1}{4\pi}.
	\]
\end{corollary}

The expression $\sqrt{s s_a s_b s_c}$ also appears in the formula for the circumradius of a triangle:
\[
	\frac{a b c}{4 \sqrt{s s_a s_b s_c}}
\]
While the expected value of the circumradius diverges, the expected value of its reciprocal -- that is, the expected curvature of the circumcircle -- does not:

\begin{corollary}\label{cor:expected circumcurvature}
	With respect to the symmetric measure on triangles with perimeter 2, the expected value of the curvature of the circumcircle is
	\[
		E(\text{circumcurvature}) = \frac{\pi}{2}.
	\]
\end{corollary}

\begin{proof}
	The expectation of circumcurvature is
	\[
		\iint_\Delta \frac{4 \sqrt{s s_a s_b s_c}}{a b c} \psi(s_a,s_b) \dArea.
	\]
	Using the definition of $\psi$ from \eqref{eq:pushforward density 2} along with $s=1$, $s_c=s-s_a-s_b=1-s_a-s_b$, and $a=1-s_a$, $b=1-s_b$, and $c=1-s_c$, this simplifies as
	\begin{multline*}
		\iint_\Delta \frac{4}{2\pi(1-s_a)(1-s_b)(s_a+s_b)}\dArea \\
		= \int_0^1 \int_0^{1-s_a} \frac{2}{\pi(1-s_a)(1-s_b)(s_a+s_b)} \,\mathrm{d}s_b\,\mathrm{d}s_a = \frac{\pi}{2}.
	\end{multline*}
	
\end{proof}
It is great fun to compute the expectation of other natural quantities in triangle geometry, and we invite the reader to continue along these lines.

\section{Coordinates for $n$-gons}
We are now going to extend our picture to $n$-gons. We will start by generalizing our previous coordinates $x$, $y$, and $z$ for triangles. Remember that the vector $\vec{p} = (x,y,z)$ was the unit normal vector to the plane in $G_2(\R^3)$ defining the triangle, or the cross product of two orthonormal vectors $\vec{u}$, $\vec{v}$ giving a basis for the plane. Each coordinate of $\vec{p}$ is the determinant of a $2 \times 2$ matrix of coordinates $(\begin{smallmatrix}
u_i & v_i \\
u_j & v_j 
\end{smallmatrix})$
from $\vec{u}$ and $\vec{v}$. The length of $\vec{p}$ depends on $\vec{u}$ and $\vec{v}$, but $\vec{p}$ will always lie on the line normal to the plane. For that reason, it is useful to think of $\vec{p}$ as defined up to scalar multiplication\footnote{the scalar is determined by $\det A^T A = \nicefrac{1}{2} \sum \Delta_{ij}^2$, where $A$ is the $n \times 2$ matrix with columns $\vec{u}$, $\vec{v}$.}. In $\R^3$, there are precisely $\binom{3}{2} = 3$ such determinants, but in $\R^n$, there are $\binom{n}{2}$ such determinants. This leads you to construct
\begin{definition}
The \emph{Pl\"ucker coordinates} on $G_2(\R^n)$ associated to the plane $P$ spanned by $\vec{u}$ and $\vec{v}$ are the skew-symmetric matrix $\Delta(P)$ given by taking all the $2 \times 2$ minor determinants of the $n \times 2$ matrix $A = \left(\vec{u}\, \vec{v}\right)$ and identifying matrices which are scalar multiples of each other. This has several (immediately) equivalent forms: \label{def:plucker}
\begin{equation}\label{eq:plucker def}
\Delta(P)_{ij} = \det
\begin{pmatrix}
u_i & v_i \\
u_j & v_j 
\end{pmatrix}
= (u_i,v_i) \times (u_j,v_j) = 
\left( A 
\left(
\begin{smallmatrix}
0 & 1 \\
-1 & 0 \\
\end{smallmatrix}
\right)
A^T \right)_{ij}.
\end{equation} 
When it's clear which plane the coordinates refer to, we'll just write $\Delta$ and $\Delta_{ij}$.

\end{definition}
 We note in passing that the Pl\"ucker matrix is skew-symmetric. The matrix $\left( \begin{smallmatrix} 0 & 1 \\ -1 & 0 \end{smallmatrix} \right)$ should be familiar here; it represents multiplication by $-\I$ in the standard matrix representation of complex numbers. In context, it defines a complex structure $J$ on $P$ given by $J(\vec{u}) = -\vec{v}$ and $J(\vec{v}) = \vec{u}$, which is to say $J$ rotates $P$ by $90^\circ$ from $\vec{v}$ to $\vec{u}$. The Pl\"ucker matrix $\Delta(P)$ now has a clear geometric interpretation: as a linear map $\Delta(P)$ orthogonally projects each $\vec{x} \in \R^n$ to $P$ and then twists by the complex structure $J$.

The definition tells us how to find the Pl\"ucker matrix from the basis $\vec{u}$, $\vec{v}$, and this geometric interpretation (or the last expression in~\eqref{eq:plucker def}) tell us how to go back. Since $\Delta(P) \vec{u} = -\vec{v}$ and $\Delta(P) \vec{v} = \vec{u}$, the pairs $(\vec{u},-\vec{v})$ and $(\vec{v}, \vec{u})$ are singular vector pairs associated to the singular value $1$ for $\Delta(P)$. This means that the two left singular vectors or the two right singular vectors corresponding to the singular value 1 also give an orthonormal basis for the plane. Hence, we can recover an orthonormal basis for the plane from the Pl\"ucker coordinates by taking the SVD of the matrix $\Delta(P)$. Since the singular values are known in advance (two singular values are 1, the rest are 0), this is constructive and exact.

Last, we remark that while every Pl\"ucker matrix is skew-symmetric, the Pl\"ucker matrices are only a small subset of the skew-symmetric matrices: the Pl\"ucker coordinates obey an interesting system of Pl\"ucker relations which encode the fact that these subdeterminants are not all independent. The super-diagonal entries of the Pl\"ucker matrix define homogeneous coordinates of an embedding of $G_2(\R^n)$ into the projective space $\RP^{\binom{n}{2}-1}$. Since $\dim G_2(\R^n) = 2(n-2)$ is less than $\dim \RP^{\binom{n}{2}-1} = \binom{n}{2}-1$ for $n \geq 4$, we expect that these coordinates satisfy additional constraints. In fact, the constraints are simple: for each choice of four distinct rows $i<j<k<\ell$ from the matrix $A = (\vec{u}\,\vec{v})$, there are six Pl\"ucker coordinates $\Delta_{ij},\Delta_{ik},\Delta_{i\ell},\Delta_{jk},\Delta_{j\ell} , \Delta_{k\ell}$ coming from the six possible $2 \times 2$ minors involving the four rows. These six coordinates must satisfy the \emph{Pl\"ucker relation}
\begin{equation}\label{eq:plucker relation}
	\Delta_{ij}\Delta_{k\ell} - \Delta_{ik}\Delta_{j\ell} + \Delta_{i\ell}\Delta_{jk} = 0.
\end{equation}
The relations coming from all possible choices of four rows define a system of homogeneous quadratic equations which exactly cut out the image of the Grassmannian inside projective space. See~\cite{Kleiman:1972wb} for a beautifully clear discussion of these matters.

\section{From continuous symmetries to discrete symmetries}
We have now shown that $O(3)$ is a transitive group which acts on triangles and preserves the measure; generalizing, Definition~\ref{def:ngons} tells us that $O(n)$ does the same for $n$-gons. This is the start of a fascinating journey, as the structure of the orthogonal group is one of the most beautiful chapters in algebra. Every math student knows that there are only 5 Platonic solids; more advanced ones know that the relatively scarcity of these extraordinary shapes comes from the fact there are only a few finite subgroups of $O(3)$. However, in higher dimensions there are many more finite subgroups, and each of these yields an beautiful symmetry of polygon space. 

We focus on the hyperoctahedral group $B_n$, which is the subgroup of matrices in $O(n)$ of signed permutations of the coordinates $x_1, \dots, x_n$; it is the group of symmetries of the hypercube and of its dual, the cross-polytope or hyperoctahedron. These are the only matrices in $O(n)$ with integer coordinates: each such matrix can be written as the product of a diagonal matrix with entries $\pm 1$ and a permutation matrix. It will be most convenient to describe an arbitrary element $\beta \in B_n$ as a permutation of $(-n, \dots, -1,1, \dots, n)$ obeying the condition $\beta(-i) = -\beta(i)$. Our goal now is to describe the action of the hyperoctahedral group on polygon space.

The action of $B_n$ on a plane $P \in G_2(\R^n)$ permutes the rows of any basis $(\vec{u}, \vec{v})$ for $P$ and changes some of their signs. This action is never effective: reversing the sign of all rows yields the same plane, though this element is the only nontrivial stabilizer of a generic plane. The action descends to an action on polygon space. A generic polygon is stabilized by the $(\Z/2\Z)^n$ subgroup of $B_n$ of signed permutations $\beta$ where $\beta(i) = \pm i$ for all $i$. The quotient group of unsigned permutations $S_n = B_n/(\Z/2\Z)^n$ simply permutes the edges of the polygon. These facts prove that

\begin{proposition}
The symmetric measure on polygon space is invariant under permutations of the edges.
\end{proposition}

%\begin{proof}
%Changing the sign of a row $(a_i,b_i)$ of $A$ does not change the edge $(a_i^2 - b_i^2, 2a_ib_i)$, but permuting the rows of $A$ permutes the edges.
%\end{proof}

We now want to explore some consequences of this invariance. 
It follows directly from the definitions that $\beta$ acts in a nice way on Pl\"ucker matrices:
\begin{proposition}
For any $\beta \in B_n$, \label{prop:plucker action} 
\begin{equation*}
\Delta(\beta P)_{ij} = \sgn \beta(i) \sgn \beta(j) \Delta(P)_{|\beta(i)|\,|\beta(j)|}.
\end{equation*}
\end{proposition}
We now start describing subsets of polygon space:
\begin{definition} 
We say that a 2-plane $P$ with orthonormal basis $A = (\vec{u}\, \vec{v})$ is a \emph{semicircular lift} of a polygon if the directions of the vectors $(u_i,v_i)$ all lie on the semicircle (oriented counterclockwise) between $(u_1,v_1)$ and $-(u_1,v_1)$. \end{definition}
We might worry that this idea is not well-defined; after all, there are many orthonormal bases for the plane! But it is easy to see that changing bases just rigidly rotates the collection of vectors $(u_i,v_i)$, which preserves the property above. 

\section{Polygons and the positive Grassmannian} A subset of the Grassmannian which has attracted a lot of interest recently in string theory~\cite{ArkaniHamed:2016ey} is the \emph{positive Grassmannian} of planes with a basis for which $\Delta_{ij}(P) > 0 \iff i < j$. We note that any basis for a plane in the positive Grassmannian has all signs in the upper triangle agreeing, but that reversing the orientation of the plane (for instance) reverses the signs of all Pl\"ucker coordinates. Therefore, we might also see matrices with all negative signs in the upper triangle. (In this case we took the wrong basis.)

This subspace has a natural meaning in our terms:

\begin{proposition} 
The positive Grassmannian $G_2(\R^n)^+$ consists of planes with a basis which is a semicircular lift of a strictly convex polygon.
\label{prop:positive grassmannian is convex polygons}
\end{proposition}

The proof is a pleasant exercise in chasing down definitions, so we leave it to the reader with this hint: strict convexity of the polygon is equivalent to the statement that the edge directions are distinct and in counterclockwise order on the circle and a semicircular lift preserves this property. We note that a very similar interpretation of the positive Grassmannian for 2-planes which shows that the positive Grassmannian has the same topology as the convex polygons appears in Section 5.3 of~\cite{ArkaniHamed:2016ey}.

%\begin{proof}
%Suppose $P$ is a semicircular lift of a convex polygon. We may assume that $a_1 + I b_1 = r_1 > 0$, and that $a_j + I b_j = r_j e^{i \theta_j}$. Since the lift is semicircular, $0 < \theta_j < \pi$. Since the polygon is convex, the edge directions $2\theta_j$ are sorted, so $\theta_i > \theta_j \iff i > j$. But then $\Delta_{ij} = r_i r_j \sin (\theta_j - \theta_i) > 0 \iff i < j$.
%
%Going backwards, suppose $P$ has a basis $\vec{a}, \vec{b}$ where $\Delta_{ij}(P) > 0 \iff i < j$. Writing $a_j + I b_j = r_j e^{i \theta_j}$ and assuming $\theta_1 = 0$ as before, since the $\Delta_{1j} = r_1 r_j \sin \theta_j > 0$, we know $0 < \theta_j < \pi$, so $P$ is a semicircular lift. On the other hand, since $\Delta_{ij} = r_i r_j \sin (\theta_j - \theta_i)$, and we now know $\theta_j - \theta_i < \pi$, we can conclude that $\sin (\theta_j - \theta_i) > 0 \iff \theta_i < \theta_j$, or that $\theta_i < \theta_j \iff i < j$. This means that the edge directions $2\theta_j$ of the polygon are sorted, and hence that the polygon is strictly convex.
%\end{proof}
%
%\begin{center}
%Note: I think that at least half of this proof doesn't actually have to be in the final draft, as it's more or less obvious. But I wrote it out just to convince myself there were no hidden surprises-- J.
%\end{center}

Since the property of being a convex polygon is invariant under cyclic permutations of the edges, we expect a cyclic subgroup of the hyperoctahedral group to preserve $G_2(\R^n)^+$. In fact, the full stabilizer is somewhat bigger, and the cyclic part is not quite what we'd expect:

\begin{proposition}
The stabilizer of $G_2(\R^n)^+$ inside the hyperoctahedral group is the subgroup of order $4n$ generated by \label{prop:stabilizer}
\begin{equation*}
\begin{aligned}
\beta  &= (1,2, \dots (n-1),-n)(-1,-2, \dots -(n-1),n), \\
\eta   &= (-1,1)(-2,2) \dots (-n,n), \\
\gamma &= (1,n)(2,n-1) \dots (-1,-n)(-2,-(n-1)) \dots
\end{aligned}
\end{equation*}
Note that the subgroup generated by $\beta$ is cyclic of order $n$, but \emph{not} the canonical cyclic subgroup of order $n$ generated by $(1,2, \dots, n)(-1,-2, \dots, -n)$.
\end{proposition}

\begin{proof}
We start by showing that all these group elements map $G_2(\R^n)^+$ to itself, using Proposition~\ref{prop:plucker action}. If $P$ is in $G_2(\R^n)^+$, then 
\begin{equation*}
\Delta(\beta P)_{ij} = \Delta(P)_{(i+1)(j+1)} > 0 \iff i < j \text{ and } i, j \neq n
\end{equation*}
while
\begin{equation*}
\Delta(\beta P)_{in} = -\Delta(P)_{(i+1)1} > 0, \text{ since } \Delta(P)_{1(i+1)} > 0,
\end{equation*}
so $\beta P$ is still in $G_2(R^n)^+$. The element $\eta$ doesn't change any Pl\"ucker coordinates:
\begin{equation*}
\Delta(\eta P)_{ij} = \sgn \eta(i) \sgn \eta(j) \Delta(P)_{ij} = (-1)(-1) \Delta(P)_{ij} = \Delta(P)_{ij}.
\end{equation*}
Thus $\eta P$ is actually the same plane! (And so it's definitely still in $G_2(\R^n)^+$.) Since $\gamma$ reverses the order of the coordinates, $i > j \iff \gamma(i) < \gamma(j)$. Further, $\Delta(P)_{ij} > 0 \iff i < j$. Thus,
\begin{equation*}
\Delta(\gamma P)_{ij} = \Delta(P)_{\gamma(i)\gamma(j)} > 0 \iff i > j.
\end{equation*}
Thus $\Delta(\gamma P)$ has positive entries \emph{below} the main diagonal and negative entries \emph{above} in some basis $(\vec{u}, \vec{v})$ for $P$. But the basis $(\vec{v}, \vec{u})$ for $P$ has opposite signs for all Pl\"ucker coordinates, and so has $\Delta_{ij} > 0 \iff i < j$, as desired. Thus $\gamma P \in G_2(\R^n)^+$.

It's fun to see this geometrically as well. Figure~\ref{fig:stabilizer} shows a convex 4-gon and its semicircular lift. If we cyclically permute the edges so that we start with edge 4, we must take the other square root of edge direction $4$ to keep the lifts of the other edges in the semicircle extending counterclockwise from $(u_4,v_4)$. This generalizes to $n$-gons.

We have now proved that our subgroup stabilizes $G_2(\R^n)^+$. But we haven't proved that it is the \emph{largest} subgroup of the hyperoctahedral group with this property. So take any hyperoctahedral group element $\delta$ which stabilizes $G_2(\R^n)^+$. We know $\delta$ must send $1$ to some $\pm k$. The product $\pi$ of $\delta$ and $\beta^{n-|k|}$ (and, if needed, $\eta$) fixes $1$. It suffices to show that $\pi$ is the identity; if so, $\delta$ was in the subgroup.

We start by proving that $\pi(j) > 0$ for all positive $j$. Suppose not. Then if $P$ is in the positive Grassmannian,
\begin{equation*}
\Delta(\pi P)_{1j} = \sgn \pi(1) \sgn \pi(j) \Delta P_{1|\pi(j)|} = -\Delta P_{1|\pi(j)|} < 0,
\end{equation*}
since $\Delta P_{1|\pi(j)|} > 0$ by our assumption that $P$ was in the positive Grassmannian.
Thus $\pi P$ is not in the positive Grassmannian, a contradiction. This means that $\pi$ consists of matching permutations of $1, \dots, n$ and $-1, \dots, -n$.

We are now going to prove by induction that $\pi(k) = k$ for all $k$. We have just established the base case ($k=1$). So suppose $\pi(k) = k$ for $k < K$, and consider $\pi(K)$. If $\pi(K) \neq K$, then $\pi(K) > K$ (since $1, \dots, K-1$ are taken), and there is some $j > K$ so that $\pi(j) = K$. But then
\begin{equation*}
\Delta(\pi P)_{Kj} = \Delta P_{\pi(K) K} < 0,
\end{equation*}
even though $\pi(K) > K$, and $\pi P$ is not in the positive Grassmannian, a contradiction. Thus $\pi(K) = K$, and we have proved that $\pi$ is the identity permutation. \end{proof}

\begin{figure}
\begin{center}
\hphantom{.}
\hfill
\begin{overpic}[width=1.5in]{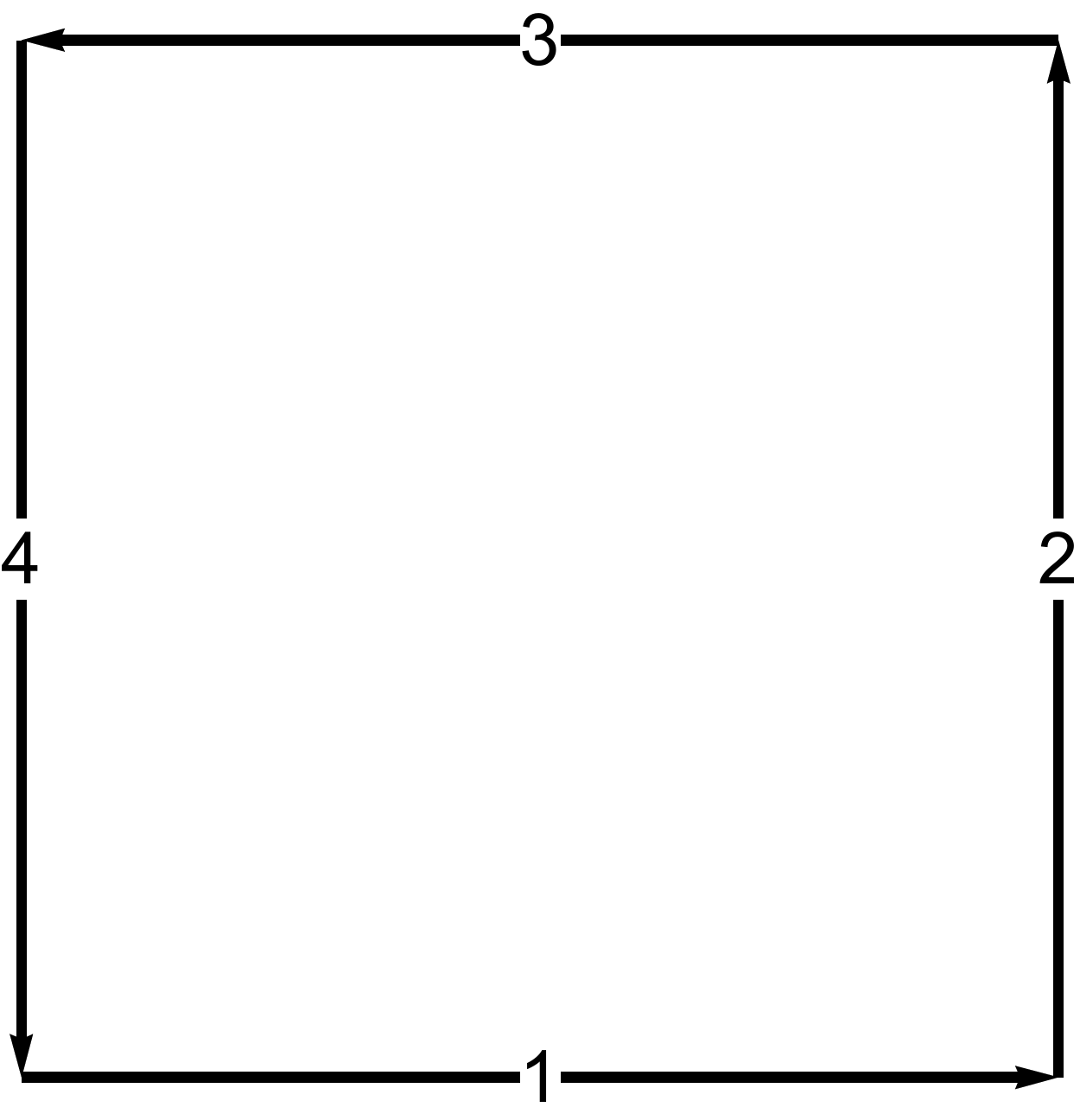}
\end{overpic}
\hfill
\begin{overpic}[width=1.5in]{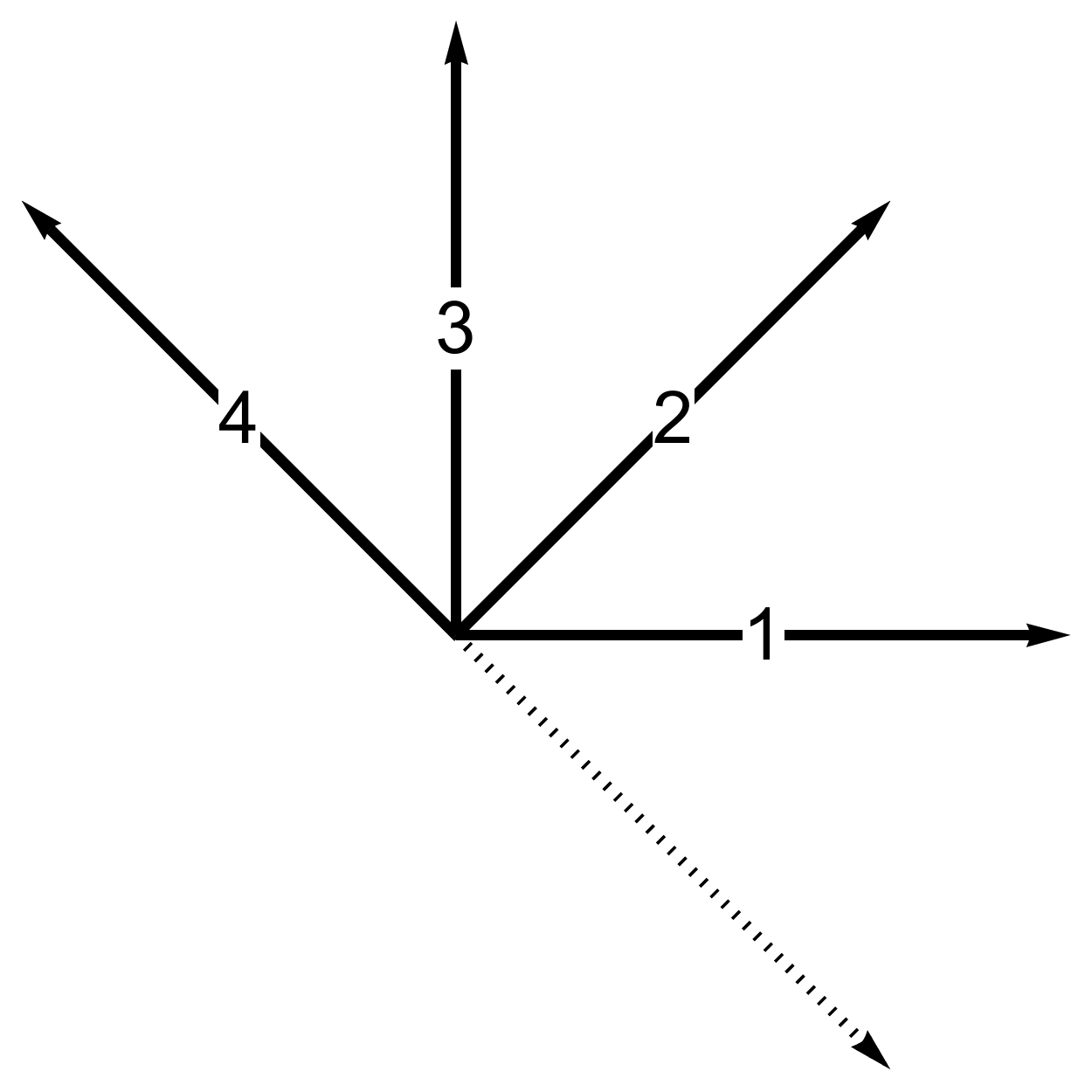}
\end{overpic}
\hfill
\hphantom{.}
\end{center}
\caption{A convex polygon in the positive Grassmannian $G_2(\R^n)^+$ (left) and its semicircular lift (right). If we cyclically permute the edges to put edge 4 first, we must take the opposite lift of edge 4 to make the lift semicircular.
\label{fig:stabilizer}}
\end{figure}
%These figures from positive-grassmannian-stabilizer.nb

We can now subdivide the space of polygons in a natural way by dividing the Grassmannian into \emph{sign chambers} by grouping together all planes for which the matrix $S_{ij} = \sgn \Delta_{ij}$ of signs of the Pl\"ucker coordinates is the same. We will call these \emph{Pl\"ucker sign matrices}. By convention, the sign chambers will be the open subsets of the Grassmannian whose Pl\"ucker sign matrices have zeros only on the diagonal. The positive Grassmannian, for instance, is the sign chamber corresponding to the Pl\"ucker sign matrix $S^0$ defined by $S^0_{ij} = 1 \iff i < j$, $S^0_{ii} = 0$. 

Like the original Pl\"ucker matrices, Pl\"ucker sign matrices are skew-symmetric and defined up to scalar multiplication (by $\pm 1$). But not every skew-symmetric matrix of $\pm 1$'s is a Pl\"ucker sign matrix -- the Pl\"ucker relations~\eqref{eq:plucker relation} rule some out. This means that it is interesting to count the sign chambers and determine whether there are different types of sign chambers or whether they are all identical.  

\begin{proposition}
The \label{prop:transitive} action of the hyperoctahedral group on $G_2(\R^n)$ descends to a transitive hyperoctahedral group action on the sign chambers and the Pl\"ucker sign matrices. 
\end{proposition}

\begin{proof}
It follows immediately from Proposition~\ref{prop:plucker action} that the action of the hyperoctahedral group on the Grassmannian and the Pl\"ucker matrices induces a corresponding action on the set of Pl\"ucker sign matrices. So suppose we have an arbitrary Pl\"ucker sign matrix $S$ corresponding to some plane $P \in G_2(\R^n)$ with a basis $(\vec{u}$, $\vec{v})$, as usual. It suffices to show that there is a hyperoctahedral group element which maps $S$ to $S^0$.

Some collection of sign changes puts all the $(u_i,v_i) = u_i + v_i\I = r_i e^{\I \theta_i}$ in the semicircle extending counterclockwise from $(u_1,v_1) = (r_1,0)$. There is then a unique permutation of $2,\dots, n$ which fixes $(u_1,v_1)$ and puts the remaining $(u_i, v_i)$ in counterclockwise order by direction. Together, the permutation and sign changes are some element $\beta$ of the hyperoctahedral group. But now the basis $\vec{u}$, $\vec{v}$ is a semicircular lift of a convex polygon and by Proposition~\ref{prop:positive grassmannian is convex polygons}, the resulting plane is in the positive Grassmannian. Hence, it has Pl\"ucker sign matrix $S^0$.
\end{proof}

We can now count and describe the sign chambers easily:
\begin{proposition}
There are $2^{n-2} \times (n-1)!$ sign chambers and corresponding Pl\"ucker sign matrices. The sign chambers are all isometric and in particular have the same volume. 
\end{proposition}
 
\begin{proof}
By the orbit-stabilizer theorem, the size of the orbit of $S^0$ is equal to the size of the hyperoctahedral group $B_n$ (namely $2^{n} \times n!$) divided by the size of the stabilizer ($4n$, by Proposition~\ref{prop:stabilizer}). But by Proposition~\ref{prop:transitive}, the orbit of $S^0$ is the entire set of Pl\"ucker sign matrices. 
\end{proof}

We now want to understand the geometric meaning of the sign chambers. 

\begin{theorem}
The convex $n$-gons consist of $2^{n-1}$ copies of the positive Grassmannian. They comprise $2/(n-1)!$ of the space of $n$-gons.
\label{thm:convex}
\end{theorem}

\begin{proof}
We already know from Proposition~\ref{prop:positive grassmannian is convex polygons}
that the positive Grassmannian consists of the semicircular lifts of the convex polygons. Therefore, any other plane corresponding to a convex polygon must be a different lift to $G_2(\R^n)$. There are $2^{n-1}$ such lifts, remembering that changing all the signs has no effect.
\end{proof}

\section{Sylvester's 4-point problem and quadrilaterals} In the same ongoing discussion in which Woolhouse posed his versions of the obtuse triangle problem, J.~J.\  Sylvester in 1864 asked for the probability that four points ``taken at random in a plane'' formed the vertices of a reentrant (embedded, but not convex) quadrilateral~\cite{Sylvester:1864vj}.\footnote{Note that there is a typo in the original statement of Sylvester's problem: he used the word ``convex'' where he meant to say ``reentrant''.} Various solutions were proposed by Cayley~\cite[footnote 64(b)]{Sylvester:1864cz}, De Morgan~\cite[pp. 147--148]{DeMorgan:1871vv}, and others, with answers including (at least) $1/4$, $35/12\pi^2$, $3/8$, $1/3$, and $1/2$~\cite{Ingleby:1866ul}. As with Caroll's problem, it soon became clear that the probability measure for the four points was an issue, with Sylvester concluding that the triangle problem and the quadrilateral problem, as posed, ``do not admit of a determinate solution''~\cite{Sylvester:1866vu}. A robust literature has grown up around the related problem of finding the probability when the points are selected from the interior of a convex body (see in particular Blaschke's remarkable result~\cite{Blaschke:1917tm} and Pfiefer's survey~\cite{Pfiefer:1989ha}). 

From our perspective, the most compelling of the original solutions to Sylvester's problem was given by the science educator, astronomer, future priest, and past Senior Wrangler James Maurice Wilson~\cite{Wilson:1866vj}, who argued that $1/3$ of quadrilaterals are reentrant by focusing on the edges of the quadrilateral rather than the vertices. Indeed, an extrapolation of his argument suggests that $1/3$ of quadrilaterals should be convex, $1/3$ reentrant, and $1/3$ self-intersecting, the same answer we will arrive at in Theorem~\ref{thm:quadrilateral classes}. 
 
We now answer Sylvester's question in our terms. We divide quadrilaterals into three classes: convex, reflex (or reentrant), and self-intersecting. We have shown (Theorem~\ref{thm:convex}) that 1/3 of the quadrilaterals are convex; the remaining 4-gons are either reflex or self-intersecting.  The boundaries between the classes consist of polygons where two edges point in the same (or opposite) directions; that is, when rows of the $n \times 2$ matrix $A$ are colinear or perpendicular. The Pl\"ucker matrix consists of cross products of these rows, and detects colinearity; we now add the matrix of dot products of rows to detect perpendicularity:

\begin{definition}
The \emph{projection matrix} associated to a plane $P$ with (orthonormal) basis given by the $n \times 2$ matrix $A$ is given by $AA^T$. This matrix orthogonally projects vectors to the plane $P$. The entries of $AA^T$ are the dot products of the rows of $A$; $(AA^T)_{ij} = (u_i,v_i) \cdot (u_j,v_j)$.
\end{definition}

It is a neat fact that the projection matrix is closely related to the Pl\"ucker matrix!

\begin{proposition}\label{prop:projection matrix plucker square}
If $A$ is an orthogonal $n \times 2$ matrix, the projection matrix 
\begin{equation*}
AA^T = -(\Delta(P))^2.
\end{equation*}
\end{proposition}

\begin{proof}
We noted in Definition~\ref{def:plucker} that $\Delta(P) = A \left( \begin{smallmatrix} 0 & 1 \\ -1 & 0 \end{smallmatrix} \right) A^T$. Expanding $\Delta(P)^2$:
\begin{equation*}
\Delta(P)^2 = A \left( \begin{smallmatrix} 0 & 1 \\ -1 & 0 \end{smallmatrix} \right) A^T A \left( \begin{smallmatrix} 0 & 1 \\ -1 & 0 \end{smallmatrix} \right) A^T = A \left( \begin{smallmatrix} 0 & 1 \\ -1 & 0 \end{smallmatrix} \right)^2 A^T = A(-I)A^T = - AA^T,
\end{equation*}
using the fact that Gramian $A^T A = I$ since the columns of $A$ are orthonormal. 

Geometrically this is similarly clear: $-\Delta(P)^2$ has the effect of projecting a vector to $P$, rotating it by $180^\circ$, and then reversing its direction. Since the last two actions cancel each other, this is just projection to $P$.
\end{proof}

We can use the Pl\"ucker sign matrix $\sgn \Delta(P)$ and the \emph{projection sign matrix} $\sgn AA^T$ to divide the Grassmannian into natural cells. We will call these ``sign cells'' for now. As before, the hyperoctahedral group $B_n$ acts on the matrices $AA^T$ and $\sgn AA^T$ just as it did on $\Delta(P)$ and $\sgn \Delta(P)$, and we will use this group action to do our computations. 

\begin{definition}
The Grassmannian is divided into a collection of cells, called \emph{sign cells} where each $P$ belongs to the subspace of all planes with the same matrices $\sgn \Delta(P)$ of signs of Pl\"ucker coordinates and $\sgn AA^T$ of signs of entries in the projection matrix for $P$. 
\end{definition}

We will now specialize to $G_2(\R^4)$ and prove some useful facts about the sign cells:
\begin{proposition}
\label{prop: sign cells}
\begin{itemize}
\item The positive Grassmannian $G_2(\R^4)^+$ is divided into 4 sign cells. 

\item The stabilizer of the positive Grassmannian from Proposition~\ref{prop:stabilizer} acts transitively on these 4 cells; hence the hyperoctahedral group acts transitively on the sign cells of $G_2(\R^4)$. 

\item The stabilizer of the ``base'' sign cell 
\begin{equation*}
\sgn \Delta = \begin{pmatrix}
 0 &  1 &  1 & 1 \\
-1 &  0 &  1 & 1 \\
-1 & -1 &  0 & 1 \\
-1 & -1 & -1 & 0 
\end{pmatrix} 
\quad
\sgn AA^T = 
\begin{pmatrix}
 1 &  1 &  1 & -1 \\
 1 &  1 &  1 & 1 \\
 1 &  1 &  1 & 1 \\
-1 &  1 &  1 & 1 
\end{pmatrix}
\end{equation*}
consists only of the 4 element group generated by 
\begin{equation*}
\begin{aligned}
\eta &= (-1,1)(-2,2) \dots (-n,n)\\
\gamma &= (1,n)(2,n-1) \dots (-1,-n), (-2,-(n-1)), \dots.
\end{aligned}
\end{equation*} 

\item There are $384 = 2^4 \times 4!$ elements of the hyperoctahedral group $B_4$ and hence $96 = 384/4$ different sign cells.

\item Each sign cell is equiprobable (in fact, each is isometric).
\end{itemize}
\end{proposition}

\begin{proof}
	Using Proposition~\ref{prop:projection matrix plucker square} and the fact that $\Delta_{ij} = -\Delta_{ji}$, we can write the projection matrix $AA^T$ in terms of the Pl\"ucker coordinates. Since $AA^T$ is symmetric and since the diagonal entries are simply the squared norms of the rows of $\Delta$ and hence positive, the projection sign matrix is completely determined by the super-diagonal triangle of $AA^T$, which is
	\begin{equation}\label{eq:projection matrix formula}
		\begin{pmatrix}
		\ast & \Delta_{13} \Delta_{23} + \Delta_{14} \Delta_{24} & \Delta_{14} \Delta_{34} - \Delta_{12} \Delta_{23} & -\Delta_{12} \Delta_{24} - \Delta_{13} \Delta_{34} \\
		 \ast & \ast & \Delta_{12} \Delta_{13}+\Delta_{24} \Delta_{34} & \Delta_{12} \Delta_{14}-\Delta_{23} \Delta_{34} \\
		 \ast & \ast & \ast & \Delta_{13} \Delta_{14}+\Delta_{23} \Delta_{24} \\
		 \ast & \ast & \ast & \ast 
		\end{pmatrix}
	\end{equation}
(the elements replaced with $\ast$ don't affect our analysis). For any element of the positive Grassmannian we know that $\Delta_{ij} > 0$ for all $i < j$, so the only entries in $AA^T$ whose signs could disagree with the base sign cell are $(AA^T)_{13}$ and $(AA^T)_{24}$. Since there are only 4 possible sign combinations for these two entries, we see that there are no more than 4 sign cells in the positive Grassmannian. 
	
	Now, we examine the action of the stabilizer of the positive Grassmannian. As we saw in Proposition~\ref{prop:stabilizer}, the element $\eta$ does not affect the Pl\"ucker coordinates, and hence must also fix the projection sign matrix. By inspection of the above expression for $AA^T$, the element $\gamma$ reflects the entries of the projection matrix across the anti-diagonal, which also has no effect on the projection sign matrix of elements of the base sign cell.
	
	The action of $\beta$ is more complicated, though straightforward enough to write down using our expression for $AA^T$. Among other things, $\beta$ replaces $(AA^T)_{13}$ with $(AA^T)_{24}$ and replaces $(AA^T)_{24}$ with $-(AA^T)_{13}$. When applied to the base sign cell, this turns out to be the only way in which $\beta$ affects the projection sign matrix. In turn, this means that $\beta^2$ replaces $(AA^T)_{13}$ with $-(AA^T)_{13}$ and replaces $(AA^T)_{24}$ with $-(AA^T)_{24}$, and that $\beta^3$ replaces $(AA^T)_{13}$ with $-(AA^T)_{24}$ and replaces $(AA^T)_{24}$ with $(AA^T)_{13}$. In particular, all 4 possible sign cells are actually realized and the cyclic subgroup generated by $\beta$ acts freely and transitively on them.
	
	Since the stabilizer of the base sign cell must be contained in the stabilizer of the entire positive Grassmannian, we have shown that the stabilizer is exactly the subgroup generated by $\eta$ and $\gamma$. The count of sign cells and the proof that they are isometric now go as they did in the proof of Proposition~\ref{prop:transitive}.
\end{proof}

We can now show:

\begin{proposition}
\label{prop: sign cells are path connected} The sign cells in $G_2(\R^4)$ are path-connected.
\end{proposition}

\begin{proof}
By Proposition~\ref{prop: sign cells}, it suffices to prove this for the base sign cell since all the sign cells are isometric. So suppose we have two planes $P_0$ and $P_1$ in the base sign cell, with Pl\"ucker coordinates $\Delta_{ij}^0$ and $\Delta_{ij}^1$. Keeping in mind that $\Delta_{ij} = -\Delta_{ji}$, we can restrict our attention to the six coordinates $\Delta_{ij}$ with $i < j$ for the remainder of the proof.  (The complementary Pl\"ucker coordinates follow a similar argument with various signs and inequalities reversed, which we will not write out.) These are all positive numbers which must obey the single Pl\"ucker relation
\begin{equation}
\Delta_{12} \Delta_{34} - \Delta_{13} \Delta_{24} + \Delta_{14} \Delta_{23} = 0.
\label{eq: plucker for quadrilaterals}
\end{equation}
Our strategy is to interpolate between $P_0$ and $P_1$. While there is certainly a geodesic path joining $P_0$ and $P_1$ which would be a natural candidate for interpolation, that path does not seem to always stay within the sign cell. Thus, we will join $P_0$ and $P_1$ by interpolating between $\Delta_{ij}^0$ and $\Delta_{ij}^1$ by a family of Pl\"ucker coordinates $\Delta_{ij}(t)$. 

We define the $\Delta_{ij}(t)$ by the logarithmic interpolation\footnote{So called because $\ln \Delta_{ij}(t)$ interpolates linearly between $\ln \Delta_{ij}^0$ and $\ln \Delta_{ij}^1$.}
\begin{equation}
\Delta_{ij}(t) = \left(\Delta_{ij}^0 \right)^{1-t} \left(\Delta_{ij}^1\right)^t
\label{eq: geometric interpolation}
\end{equation}
except for $\Delta_{24}(t)$, which must be given by   
\begin{equation*}
\Delta_{24}(t)  = \frac{\Delta_{12}(t) \Delta_{34}(t) + \Delta_{14}(t) \Delta_{23}(t)}{\Delta_{13}(t)}
\end{equation*}
in order to ensure that the $\Delta_{ij}(t)$ obey the Pl\"ucker relation~\eqref{eq: plucker for quadrilaterals}.  Since the $\Delta_{ij}^0$ and $\Delta_{ij}^1$ are positive, it is straightforward to check that all the $\Delta_{ij}(t)$ are also positive.

We must now prove that the $(AA^T)_{ij}(t)$ have the correct signs. As in~\eqref{eq:projection matrix formula}, we can write each  $(AA^T)_{ij}(t)$ in terms of the $\Delta_{ij}(t)$ and, since $\Delta_{ij}(t) > 0$, only 
\begin{align*}
\left(AA^T\right)_{13}(t) &= -\Delta_{12}(t) \Delta_{23}(t) + \Delta_{14}(t) \Delta_{34}(t) \text{ and} \\
\left(AA^T\right)_{24}(t) &= \Delta_{12}(t) \Delta_{14}(t) - \Delta_{23}(t) \Delta_{34}(t)
\end{align*}
could change sign. So it suffices to show $\left(AA^T\right)_{13}(t) > 0$ and $\left(AA^T\right)_{24}(t) >0$ knowing that these inequalities are satisfied for $t=0$ and $t=1$. Rearranging, this is equivalent to showing that for all $t$,
\begin{equation}
\frac{\Delta_{14}(t)}{\Delta_{23}(t)} > \frac{\Delta_{12}(t)}{\Delta_{34}(t)} \quad\text{and}\quad
\frac{\Delta_{14}(t)}{\Delta_{23}(t)} > \frac{\Delta_{34}(t)}{\Delta_{12}(t)}.
\label{eq: quotient inequalities}
\end{equation} 

The form of these inequalities explains why we chose logarithmic interpolation rather than linear interpolation: linearly interpolating the numerator and denominator of a fraction has a complicated effect on the quotient, while logarithmically interpolating the numerator and denominator logarithmically interpolates their quotient. So
\begin{equation*}
\frac{\Delta_{ij}(t)}{\Delta_{kl}(t)} = \frac{\left(\Delta_{ij}^0\right)^{1-t} \left(\Delta_{ij}^1\right)^t}{\left(\Delta_{kl}^0\right)^{1-t} \left(\Delta_{kl}^1\right)^t} = \left( \frac{\Delta_{ij}^0}{\Delta_{kl}^0} \right)^{1-t} \left(\frac{\Delta_{ij}^1}{\Delta_{kl}^1} \right)^t.
\end{equation*}
This means that both sides of the inequalities in~\eqref{eq: quotient inequalities} are being logarithmically interpolated between values at $t=0$ and $t=1$ where the inequalities are obeyed. It is a general fact about logarithmic interpolation that this implies the inequalities are satisfied for intermediate values of $t$. Proving this is a fun exercise: start by taking the logarithm of the inequalities in~\eqref{eq: quotient inequalities}.
\end{proof} 

 We can now see that the sign cells have geometric meaning at the level of quadrilaterals:

\begin{proposition}\label{prop: sign cell determines class}
All of the quadrilaterals in any given sign cell are either convex, reflex, or self-intersecting.
\end{proposition}

\begin{proof}
The walls between sign cells consist of planes with a Pl\"ucker coordinate or entry in the projection matrix equal to zero. Recall that the edges of the polygon $e_i$ are equal to the squares of the complex numbers $z_i = u_i + \I v_i$. 

When the Pl\"ucker coordinate $\Delta_{ij} = 0$, two rows $(u_i,v_i)$ and $(u_j,v_j)$ are colinear, and $z_i = \lambda z_j$ for some real $\lambda$. Squaring, we see that $e_i = \lambda^2 e_j$, and the edges $e_i$ and $e_j$ point in the same direction.

When the projection coordinate $(AA^T)_{ij} = 0$, two rows $(u_i,v_i)$ and $(u_j,v_j)$ are perpendicular, and $z_i = \lambda \I z_j$ for some real $\lambda$. Squaring, we see that now $e_i = -\lambda^2 e_j$, and the edges $e_i$ and $e_j$ point in opposite directions.

Since we have showed that the sign cells are path-connected in Proposition~\ref{prop: sign cells are path connected}, the polygons in each sign cell can be deformed to one another without making any pair of edge directions parallel or antiparallel.  But to transition between convex, reflex, and self-intersecting, two adjacent edges must point in the same or opposite directions.
\end{proof}

We pause for a minute to appreciate the significance of this result. We now know that $G_2(\R^4)$ is broken up into 96 isometric pieces, each of which corresponds to a collection of quadrilaterals which are all convex, all reflex, or all self-intersecting. This means that we can determine which category \emph{all} elements of a given sign cell are in by choosing a \emph{particular} element of the sign cell and determining whether it is convex, reflex, or self-intersecting.

Since each of the sign cells is the image of the base sign cell under the action of 4 different elements of the hyperoctahedral group $B_4$ (namely, conjugates of the stabilizer of the base sign cell), we can choose this preferred element of each sign cell by choosing a preferred element of the base sign cell and then moving it around by hyperoctahedral elements.

In other words, to see what fraction of quadrilaterals are convex, reflex, and self-intersecting, we simply need to look at the 96 images of the special quadrilateral in the base sign cell and determine which fraction fall into each category. The answer is simple and pleasant:

\begin{theorem}\label{thm:quadrilateral classes}
Of the 96 sign cells in $G_2(\R^4)$, convex, reflex, and self-intersecting quadrilaterals each compose $32$ cells. Hence, the probabilities that a randomly selected quadrilateral is convex, reflex, or self-intersecting are each equal to $1/3$.
\end{theorem}

\begin{proof}
	As discussed above, it suffices to choose an element of the base sign cell and examine its images under the action of the hyperoctahedral group. While this will produce points in each of the 96 different sign cells, many of the resulting quadrilaterals will be identical: after all, $B_4 \simeq (\Z/2)^4 \rtimes S_4$, but the normal subgroup $(\Z/2)^4$ corresponds to different lifts of the quadrilateral, which give different points on the Grassmannian that map to the same quadrilateral. 
	
	Hence, we can simply act on our chosen element of the base sign cell by the quotient $S_4$, which just acts by permuting the edges. Since each of the 96 quadrilaterals we are tabulating is identical to one of these 24 permutation images, the theorem will follow if we see that 8 are convex, 8 are reflex, and 8 are self-intersecting.
	
	Since we can generate random points in $G_2(\R^4)$, it is easy to pick a random element of the base sign cell. But the most beautiful such choice would be a ``center of mass'' or ``average'' plane in the sign cell. 
	
	(We pause for a moment to reflect on the fact that before we started, the project of defining the average of a collection of quadrilaterals would have been rather daunting. But since we are working on the Grassmannian, we have a powerful collection of tools adapted to exactly these problems!)

	One definition of an average of a finite collection of subspaces of $\R^n$ is given by the \emph{flag mean} of the points. Given $n \times 2$ orthonormal matrices $A_1, \dots, A_m$ giving bases for the subspaces, a basis for the 2-dimensional flag mean subspace is given by the two (left) singular vectors of the $n \times 2m$ matrix $(A_1 \dots A_m)$. The flag mean is used in signal processing, and has beautiful mathematical properties; see~\cite{Marrinan:2014kf} and \cite{Draper:2014gu}.

We computed the flag mean of 10,420 points sampled uniformly from the base sign cell to be the quadrilateral with edge vectors approximately $(0.33, -0.59)$, $(-0.29, -0.13)$, $(-0.30, 0.11)$, and $(0.26, 0.62)$. This quadrilateral and its 23 companions given by permuting the edges form a collection of ideal quadrilaterals representating each of the sign cells. We complete the proof by presenting these quadrilaterals in Figure~\ref{fig:24 quadrilaterals}. The reader can easily verify that 8 are convex, 8 are reflex, and 8 are self-intersecting. 
\end{proof}

\begin{figure}[htbp]
	\centering
		\includegraphics[scale=0.5]{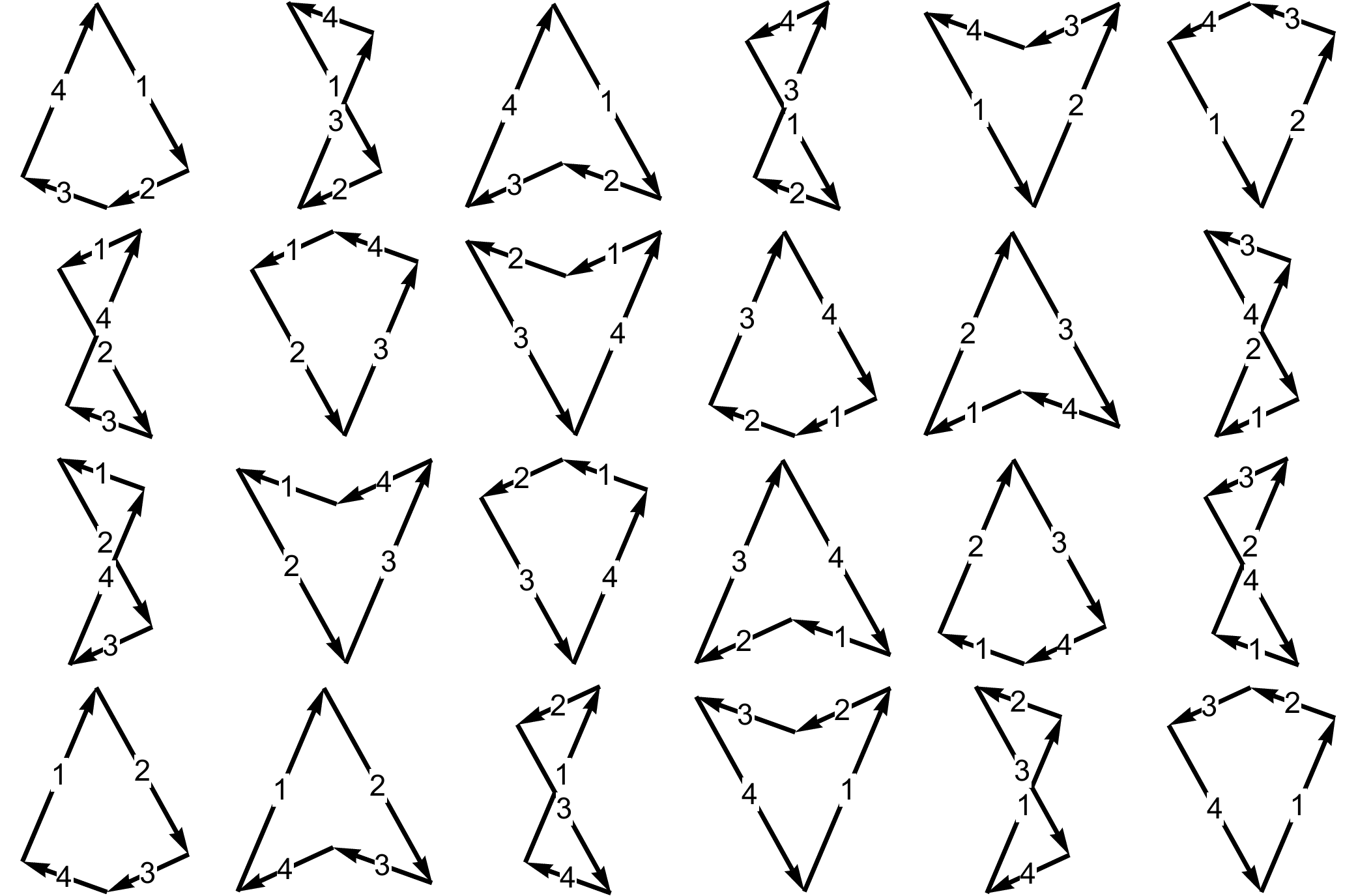}
	\caption{The permutation group orbit of the quadrilateral corresponding to the flag mean of the base sign cell. The flag mean of every sign cell corresponds to a quadrilateral congruent to one of these 24 ordered quadrilateral. By inspection, $1/3$ are convex, $1/3$ are reflex, and $1/3$ are self-intersecting.}
	\label{fig:24 quadrilaterals}
\end{figure}

Figure~\ref{fig:24 quadrilaterals} is a bit disappointing at first: despite our promise of 24 different quadrilaterals there are really only 3 that are geometrically different, one in each class. But this is easily explained: cyclically permuting the edges of a quadrilateral or reversing their order produces an ordered quadrilateral which is congruent to the original, so the standard copy of the dihedral group $D_8$ inside $S_4$ produces congruent quadrilaterals. The fact that we see 3 geometrically distinct quadrilaterals thus boils down to the fact that $|S_4|/|D_8| = 3$.

We leave it to the reader to check that the order 3 subgroup $A_3 = \{1,(123),(132)\}$ is a transversal of the dihedral group $D_8$ inside $S_4$, and hence that the three geometrically distinct quadrilaterals in Figure~\ref{fig:24 quadrilaterals} can be obtained by applying these three permutations to our chosen representative of the base sign cell.

We have now given our solution to Sylvester's problem. But notice that we have actually done much more: we have given an explicit geometry to the space of unordered (length 2) quadrilaterals by identifying the convex, reflex, and self-intersecting quadrilaterals as the three isometric Riemannian manifolds with boundary comprising the $A_3$-orbit of the base sign cell, each a Riemannian submanifold of $G_2(\R^4)$.

\section{Conclusion and Open Questions}
Our journey has taken us through some beautifully concrete applications of the Grassmannian picture of planar polygons, but there is a still a vast landscape to explore. Even for triangles, there are still a number of interesting questions open. First, it would be really interesting to be able to characterize the distance between triangles (and the effect of an arbitrary rotation of the sphere) directly in terms of triangle geometry. It is clear that you can write a rotation of the sphere in terms of a conserved quantity which you can express in terms of sidelengths of the triangle (we leave the exercise to the reader). But what does this formula mean? Second, it is tempting to go back and reprove many of the standard algebraic identities connecting various measurements of the triangle in terms of these variables (see, for instance, \cite{8680145718930201} for a trove of such identities concerning the inradii and exradii).

For quadrilaterals, a number of interesting open questions remain. The space $G_2(\R^4)$ has an interesting involution: take each plane to the perpendicular one. This gives rise to an involution on quadrilaterals which seems fascinating to explore. Many of our statements about the structure of $G_2(\R^4)$ should generalize to $G_2(\R^n)$. For instance, it seems clear that the sign cells of $G_2(\R^n)$ should be path-connected. Can you prove it? 

This metric on plane polygons can be extended to a corresponding metric on plane curves, which is used in shape recognition and classification. The paper of Younes et al.~\cite{Younes:2008gy} is a great place to start reading about this topic. The flag mean, and other tools from signal processing, also seem to have fruitful applications in polygon space. For starters: can you define the flag mean of an arbitrary subset of the Grassmannian by integration rigorously? If so,  is the flag mean of the base sign cell the kite we show above?

Several authors, notably Hausmann and Knutson~\cite{Knutson:2_iyExxE}, Kapovich and Millson~\cite{Kapovich:1996p2605}, and Howard et al.~\cite{Howard:2008uy}, have extended this structure to space polygons, and specialized it to polygons with fixed edgelengths. In our own papers~\cite{Cantarella:2014bl,Cantarella:2016bt,Cantarella:2012vn,Cantarella:2013wl,Needham:2016tg,Needham:2017vn} we have developed the sampling and integration theory for these polygon spaces and extended the theory to space curves. Space polygons of fixed edgelength form a symplectic manifold, which begins another long and fascinating story. Interestingly, that manifold seems to genuinely have more structure than the space of \emph{planar} polygons with fixed edgelengths, which remains somewhat mysterious. We hope to address fixed edgelength planar polygons in more detail in a future publication. 

\section*{Acknowledgments}
We are grateful to the MAA for the invitation to present some of this paper as an invited address at the 2017 Joint Math Meetings and to the Simons Foundation for their support of Cantarella and Shonkwiler. In addition, we'd like to thank the many colleagues who have helped us understand Grassmannians and the triangle problem, including Harrison Chapman, Rebecca Goldin, Ben Howard, Chris Manon, John McCleary, Chris Peterson, Stu Whittington, and Seth Zimmerman.

\bibliography{papers-export,triangles-special}

\end{document}